\documentclass[notitlepage]{amsart}

\usepackage{amsfonts}
\usepackage[latin1]{inputenc}
\usepackage[all]{xy} 
\usepackage{latexsym,amssymb,amscd}
\usepackage{amsmath}

\usepackage{color}

\definecolor{Myblue}{rgb}{0.0,0,0.9}
\definecolor{Mygreen}{rgb}{0.2,1,0}

\newcommand{\IP}[1]{{\color{black}#1}}
\newcommand{\MIP}[1]{{\color{black}#1}}

\newtheorem{pro}{Proposition}[section]
\newtheorem{teo}[pro]{Theorem}
\newtheorem{defi}[pro]{Definition}
\newtheorem{lem}[pro]{Lemma}
\newtheorem{cor}[pro]{Corollary}
\newtheorem{rk}[pro]{Remark}

\newtheorem{ex}[pro]{Example}

\newcommand{\I}{\mathbb{I}}

\newcommand{\N}{\mathbb{N}}
\newcommand{\T}{\mathbb{T}}
\newcommand{\Puno}{\mathbb{P}_{1}}
\newcommand{\Pinf}{\mathbb{P}_{\infty}}
\newcommand{\Iuno}{\mathbb{I}_{1}}
\newcommand{\Iinf}{\mathbb{I}_{\infty}}
\newcommand{\tla}{\tau_{\LA}}
\newcommand{\ta}{\tau_{\id}}
\newcommand{\modu}{\mathrm{mod}}

\newcommand{\Ext}{\mathrm{Ext}}
\newcommand{\Hom}{\mathrm{Hom}}
\newcommand{\End}{\mathrm{End}}
\newcommand{\Ker}{\mathrm{Ker}}
\newcommand{\Coker}{\mathrm{Coker}}

\newcommand{\pd}{\mathrm{pd}}

\newcommand{\phinjd}{\phi_r\mathrm{dim }}
\newcommand{\injd}{\mathrm{inj dim }}

\newcommand{\phd}{\phi_l\mathrm{dim }}
\newcommand{\id}{\mathfrak{A}}
\newcommand{\gld}{\mathrm{gld}}
\newcommand{\LA}{\Lambda/\mathfrak{A}}

\newcommand{\add}{\mathrm{add}}

\usepackage{latexsym,amssymb,amscd}
\usepackage{amsmath}

\begin{document}
\title[Idempotent ideals and Igusa-Todorov functions]{\IP {Idempotent ideals and \\ the Igusa-Todorov functions}}
\author{A. Gatica, M. Lanzilotta, M. I. Platzeck}
\thanks{
The authors thank the financial support received from Universidad de la Rep\'ublica, Montevideo, Uruguay, from Universidad Nacional del Sur, Bah\'{\i}a Blanca, Argentina and from CONICET, Argentina}
\begin{abstract} Let $\Lambda$ be an artin algebra and $\id$  a two-sided idempotent ideal of $\Lambda$, that is,  $\id$ is the trace of a projective $\Lambda$-module $P$ in $\Lambda$. We consider the categories of finitely generated modules over the associated rings $\LA, \Lambda$ and $\Gamma=\End_{\Lambda}(P)^{op}$ and study the relationship  between their homological properties via the Igusa-Todorov functions.
\end{abstract}

\maketitle

\section{Introduction} 
Throughout this paper we assume that $\Lambda$ is an artin algebra and all $\Lambda$-modules are in $\modu\Lambda$, the category of finitely generated left $\Lambda$-modules.

In \cite{IT} Igusa and Todorov introduced two functions $\phi$ and $\psi$ which turned out to be  powerful tools to study the finitistic dimension of some classes of algebras. 
On the other hand, associated to an idempotent ideal $\id$ of $\Lambda$,  there is an exact sequence of categories  $\modu\LA \rightarrow \modu\Lambda \xrightarrow {e_P}  \modu\Gamma$, where P is a projective module such that $\id = \tau_P\Lambda$ is the trace of $P$ in $\Lambda$, \IP{ $\Gamma = \End_\Lambda(P)^{op} $ and $e_P = \Hom_\Lambda (P,-)$ is the evaluation functor}. In \cite{APT} the authors studied the relation between the homological properties of the three categories involved: $\modu\LA , \modu\Lambda$ and  $\modu\Gamma.$ Our objective in this paper is to study the behaviour of the Igusa-Todorov functions  in this situation. For a finitely generated  $\Lambda$-module $X$, we will denote $\phi(X)$ by $\phi_l^\Lambda(X)$, and the supremum of these  numbers  for $X$ in $\modu \Lambda$ \IP{is the $\phi_l$ dimension of $\Lambda $}, denoted by  $\phd (\Lambda).$ \IP{Additionally, add$X$ denotes the full subcategory of $\modu\Lambda$ consisting of summands of finite direct sums of $X$.}

First we consider the inclusion of $\modu\LA$ in $\modu\Lambda$. To compare the values of  the Igusa-Todorov functions  in  a $\LA$-module $X$ in both categories we need the further assumption that the idempotent ideal $\id$ is a strong idempotent ideal, in the sense defined in \cite{APT}. We recall that the ideal   $\id$ is a {\it strong idempotent ideal} if the morphism $\Ext^i_{\LA}(X,Y) \rightarrow  \Ext^i_\Lambda(X,Y)$ induced by the canonical isomorphism $\Hom_{\LA} (X,Y) \rightarrow \Hom_\Lambda (X,Y)$ is an isomorphism for all $i\geq 0$ and all $X,Y$ in $\modu \LA$. We prove that  $\phi^{\LA}_l(X)\leq \phi^{\Lambda}_l(X) \, \,$  for all $ \, \, X \in \modu \LA$, whenever $\id$ is a strong idempotent ideal of finite projective dimension. Thus in this case the $\phi_l$ dimension of $\LA $ is bounded by the $\phi_l$ dimension of $\Lambda$.






In order to compare the behaviour of the  Igusa-Todorov functions under the functor $\modu\Lambda \xrightarrow {e_P}  \modu\Gamma$, 
we recall that  $e_P$ induces an equivalence between the full subcategory  of $\modu\Lambda$ consisting of the $\Lambda$-modules $X$ having a presentation in add$P$, and $\modu\Gamma$. We prove that both functions  $\phi$  and  $\psi$ are preserved   under $e_P$ for modules having a resolution in add$P$. As a consequence we obtain that when all $\Lambda$-modules with a presentation in add$P$ have also a resolution in add$P$, then $\phd\Gamma\leq\phd\Lambda$ and ${\psi_l} {\rm dim}\Gamma \leq {\psi_l} {\rm dim} \Lambda$ (Theorem \ref{puno=pinfinito}).

Then we obtain information about the $\phi$ dimension of $\Lambda$ from the $\phi$ dimensions of the algebras $\LA$ and $\Gamma$. We prove several inequalities, which are interesting when either the global dimension of $\LA$ or the global dimension of $\Gamma$ are finite.


To prove these results we use, in one hand, the characterization of the Igusa-Todorov function $\phi$ in terms of the bifunctor $\Ext (-,-)$ given in \cite{FLM}. On the other hand, the full subcategory   $\T$ of mod$\Lambda$  introduced in \cite{APT} consisting of the modules $T$ such that $\Ext^{i}_{\Lambda}(\LA ,T)=0$ for  all $i\geq 1$,  is very useful for our purposes. Consider the full subcategories   $\mathbb P_0$   and $\mathbb P_\infty$ of  $\modu\Lambda$, where $\mathbb P_0$ consists  of the modules whose projective cover is in add$P$, and $\mathbb P_\infty$  of those having a  projective resolution in add$P$. We  use the fact, proven in section 3, that  ($\mathbb P_0 ,\modu\LA$) is a torsion pair in $\modu\Lambda$ whose properties  are inherited  by the pair ($\mathbb P_\infty, \modu\LA$) in the category $\tilde\T$  dual of $\T$.








\section{Preliminaries}

\IP{Let $ \Lambda$ be an artin algebra, $M$ and $N$ in  $\modu\Lambda$. We denote by $\tau_MN$ the trace of $M$ in $N$, that is, the submodule of $N$ generated by the homomorphic images of maps from $M$ to $N$. Moreover, $P_0(M)$,  $I_0(M)$ denote the projective cover and injective envelope of $M$, and $\Omega^n(M)$, $\Omega^{-n}(M)$ the $n^{th}$ syzygy and the $n^{th}$ cosyzygy of $M$, respectively. Finally, $\pd M$ denotes the projective dimension of M and $\gld \Lambda$ stands for the global dimension of $\Lambda$.}

\IP{We start by recalling some definitions and results from \cite{APT} which will be used throughout the paper.} Let $\id$ be an idempotent ideal of $\Lambda$, $P_0$  the projective cover of $\id$, and  $P=\Lambda e$  where $e$ is an idempotent element of $\Lambda $ such that $\add P = \add P_0$. Then  $\id  = \Lambda e \Lambda = \tau_P \Lambda$ is the trace of $P$ in $\Lambda$,  mod$\,\LA$ is a Serre subcategory of mod$\Lambda$ and this  inclusion induces an exact sequence of categories  $\modu\LA \rightarrow \modu\Lambda \xrightarrow {e_P}  \modu\Gamma$, where $\Gamma = \End_\Lambda(P)^{op} $ and $e_P = \Hom_\Lambda (P,-)$ is the evaluation functor. 

Let $P/rP \simeq S_1 \oplus\cdots \oplus S_r$, with $S_i$ simple for all $i$, so that $P= P_0(S_1 \oplus\cdots \oplus S_r)$, and let $I = I_0(S_1 \oplus\cdots \oplus S_r)$. To compare the homological properties of mod$\Lambda$ and mod$\Gamma$,  full subcategories $\mathbb P_k$ and $\mathbb I_k$ were introduced in \cite{APT} for any $k\geq 0$. \IP{These subcategories  will be useful for our purposes, and are defined} as follows: $\mathbb P_k$ is the full subcategory of $\modu \Lambda$ consisting of the $\Lambda$-modules $X$ having a projective resolution $ \cdots \rightarrow P_1 \rightarrow P_0 \rightarrow X \rightarrow 0 $ with $P_i$ in $\add P$ for $ 0 \leq i \leq k$. The full subcategory $\mathbb I_k$ is defined dually.

Then  $\Hom_\Lambda(P,-)$ induces equivalences $\Puno \rightarrow \modu \Gamma$ and $\Iuno \rightarrow \modu \Gamma$. Moreover, the morphism of connected sequences of functors $\Ext^i_\Lambda(X,Y) \rightarrow  \Ext^i_\Gamma((P,X),(P,Y))$
induced by   $\Hom_\Lambda(P,-)$  is an isomorphism for $i=1,\cdots , k$, whenever $X\in \mathbb P_{(k+1)}$ or 
$Y\in \mathbb I_{(k+1)}$ (Theorem 3.2, \cite{APT}).

\IP{ We next turn our attention to the definition of the Igusa-Todorov functions, defined in \cite{IT}. Let $K_0$ denote the  abelian group generated by all symbols $[M]$, where $M$  in $\modu\Lambda$, modulo the relations a) $[C] = [A]+ [B]$ if $C\simeq A \oplus B$ and b)  $[P] = 0 $ if $P$ is projective. That is, $K_0$ is the free abelian group generated by the isomorphism classes of indecomposable finitely generated nonprojective $\Lambda$-modules. Let $\Omega :K_0 \rightarrow K_0$ denote the group homomorphism induced by the syzygy, that is, $\Omega([M]) := [\Omega(M)]$, and let $<{\rm add}M>$ be the subgroup of $K_0$ generated by the indecomposable sumands of $M$. When we apply the homomorphism $\Omega$ to this subgroup the rank does not increase: rank  $\Omega(<{\rm add}M>)\leq $ rank $<{\rm add}M>$, and  there is then an integer $n$ such that $\Omega:  \Omega^s(<{\rm add}M>) \rightarrow \Omega^{s+1}(<{\rm add}M>)$  is an isomorphism for all $s\geq n$. Then the Igusa-Todorov functions $\phi$ and $\psi$ are defined as follows:  $\phi (M) $ is the  smallest non-negative  integer  $n$ with this property, and $\psi(M):= \phi(M) + {\rm sup} \{ \pd X |\,   X \mbox { is a direct summand   of }  \Omega^{\phi(M)}(M)\  \mbox { with }  \pd X < \infty \}$.
 Since we will need also the dual notions, we will denote the Igusa-Todorov functions $\phi$ and $\psi$ by  $\phi_l$ and $\psi_l $, respectively. Using the cozyzygy  we can define    $\phi_r(M)$ and $\psi_r(M)$ in an analogous way.  Then $\phi_r(M)= \phi_l(DM)$ and $\psi_r(M)= \psi_l(DM)$, for any $M$ in $\modu\Lambda$.

Let $\phd \Lambda= \,  $sup$ \{ \phi_l(M) | \, M $ in $ \modu\Lambda\}$ and $\psi_l{\rm dim}\Lambda= \, $sup$ \{ \psi_l(M) | \, M $ in $ \modu\Lambda\}$. Moreover, for a subcategory $\mathcal X$ of $\modu\Lambda$ we indicate by $\phd \mathcal X$ and  $\psi_l{\rm dim}\mathcal X$ the supremum of the sets $ \{ \phi_l(X) | \, X $ in $ \mathcal X\}$ and $ \{ \psi_l(X) | \, X $ in $ \mathcal X\}$, respectively. Analogous notions are defined for $\phi_r$ and $\psi_r$.

We will also need the characterization of the function $\phi$ in terms of the bifunctor $\Ext^i_\Lambda (-,-)$ given in \cite{FLM}. We recall first that a pair $(X,Y)$ of objects in add$M$ is called  {\it d-division
of $M$} if the following three conditions hold:
\begin{itemize}

\item[(a)] $\add (X) \cap \add (Y) = 0$

\item[(b)]  $\Ext^d_\Lambda (X,-) \not\simeq \Ext^d_\Lambda (Y,-)$ in $\modu\Lambda$

\item[(c)] $\Ext^{d+1}_\Lambda (X,-) \simeq \Ext^{d+1}_\Lambda (Y,-)$ in $\modu\Lambda.$
\end{itemize}

 Dually, a pair $(X,Y)$ of objects in add$M$ is called  {\it d-injective division
of $M$} if (a) and the following two  conditions hold:
\begin{itemize}

\item[(b')] $\Ext^d_\Lambda (-,X) \not\simeq \Ext^d_\Lambda (-,Y)$ in $\modu\Lambda$

\item[(c')] $\Ext^{d+1}_\Lambda (-,X) \simeq \Ext^{d+1}_\Lambda (-,Y)$ in $\modu\Lambda.$
\end{itemize}

Then $\phi_l(M) = $max $(\{d \in  {\mathbb N} : $there is a d-division of $ M\} \cup \{0\})$ (\cite{FLM}, Theorem 3.6), and
$\phi_r(M) = $max $(\{d \in  {\mathbb N} : $there is a d-injective division of $ M\} \cup \{0\}).$

}

\section{Torsion theories associated to an idempotent ideal}
It is interesting to notice that the idempotent ideal $\id$ determines two torsion pairs  $(\modu\LA, \mathbb I_0)$ and $(\mathbb P_0, \modu\LA)$  in mod$\Lambda$, in the sense defined by Dickson in \cite{D}, as we state in the following proposition. 

\begin{pro} Let $\id$ be an idempotent ideal of $\Lambda$, $\id = \tau_P \Lambda$, where $P$ is 
a projective $\Lambda$-module. Then
\begin{itemize}
\item[(a)] $(\modu\LA, \mathbb I_0)$ is a torsion pair in  {\rm mod}$\Lambda$ .

\item[(b)] 
$(\mathbb P_0, \modu\LA)$ is  a torsion pair in  {\rm mod}$\Lambda$.
\end{itemize}
\begin{proof}
(a) 
To prove this we observe that $\mathbb I_0$ consists of the modules with socle in $\add (S_1 \oplus \cdots \oplus S_r)$, and that a $\Lambda$-module $M$ is in mod$\LA$ if and only if $S_1, \cdots , S_r$ are not composition factors of $M$. Then $\Hom_\Lambda (M,Y) = 0$ for $M\in \modu \LA$ and $Y\in \mathbb I_0$. Moreover, if $\Hom_\Lambda (M,Y) = 0$ for all $Y\in \mathbb I_0$, then in particular $\Hom_\Lambda (M,I_0(S_1 \oplus\cdots \oplus S_r)) = 0$, so that $S_1, \cdots , S_r$ are not composition factors of M, and  $M$ is thus a $\LA$-module. 
Finally, suppose that $\Hom_\Lambda (M,Y) = 0$ for each  $M\in \modu \LA$. Then $\Hom_\Lambda (S,Y) = 0$ for any simple $S$ not isomorphic to $S_1, \cdots , S_r$. Thus the only simples in the socle of $Y$ are amongst $S_1, \cdots, 
 S_r$, and therefore $Y \in \mathbb I_0$. This shows that $(\modu\LA, \mathbb I_0)$ is a torsion pair in mod$\Lambda$. 
 
 The statement (b) follows by duality. 
\end{proof}

\end{pro}

\vskip .1in

In the sequel we will  consider the full subcategory $\T$ of mod$\Lambda$  introduced and studied in section 5 of \cite{APT}, consisting of the modules $T$ such that  the group $\Ext^{i}_{\Lambda}(\LA ,T)=0$ for  all $i\geq 1$. Dually, we define  the subcategory $\tilde \T= D(\T_\Lambda)$ consisting of the $\Lambda$-modules $X$ such that  $\Ext^{i}_{\Lambda}(X, D(\LA_\Lambda))=0$ for all $i\geq 1$.


The notion of torsion pairs in abelian categories defined by Dickson  was extended to pretriangulated categories by Beligianis and Reiten (see \cite{BR}, Ch. II, Definition 3.1). Additive categories with kernels and cokernels are  examples of pretriangulated  categories, as shown in section 1, Example 2 of the same paper, and in this case torsion pairs are defined as follows.

\begin{defi} {\rm(\cite{BR})} A pair of subcategories $(\mathcal{X}, \mathcal{Y})$ in an additive category $\mathcal{C}$ with kernels and cokernels and closed under isomorphisms is a torsion pair if the following conditions hold:

T1) $\Hom_{\mathcal{C}}(X,Y)=0$ for all $X \in \mathcal{X}$, $Y \in \mathcal{Y}$

T2) For every $ C \in \mathcal{C}$ there is an exact sequence
$ 0 \rightarrow X_C \rightarrow C \rightarrow Y_C \rightarrow 0$ with $X_C \in \mathcal{X}$, $Y_C \in \mathcal{Y}$ (glueing sequence).

\end{defi}

We observe that for the torsion pairs 
 $(\modu\LA, \mathbb I_0)$ and $(\mathbb P_0, \modu\LA)$ above considered the glueing sequences for a module $X $ in mod$\Lambda$ are
 $0 \rightarrow \tau_{\LA} X \rightarrow X \rightarrow  X/\tau_{\LA} X \rightarrow 0$  and $0 \rightarrow \tau_\id X \rightarrow X \rightarrow  X/\tau_\id X \rightarrow 0$ respectively.

We now   turn our attention to the subcategories $\T$ and $\tilde \T$ of  $\modu \Lambda$, and study the restriction  of these torsion pairs to   $ \T$ and $\tilde \T$ respectively, under the assumption that the ideal $\id$ is strong idempotent.


\begin{pro}
Let $\id$ be a strong idempotent ideal of $\Lambda$. Then
\begin{itemize}

\item[(a)] $\mathbb I_{\infty}=\mathbb I_{0} \cap \T$, and the pair $(\modu\LA,I_\infty )$ of subcategories of $\T$ satisfies conditions T1) and T2) of Definition 3.2 of torsion pair. Moreover,   for  $T$ in $\T$ the glueing sequence is  $\, \,  0 \rightarrow \tau_{\LA} T \rightarrow T \rightarrow  T/\tau_{\LA} T \rightarrow 0$.

\item[(b)] $\mathbb P_{\infty}=\mathbb P_{0} \cap \tilde\T$, and the pair $(\modu\LA,P_\infty )$ of subcategories of $\tilde\T$ satisfies conditions T1) and T2) of Definition 3.2 of torsion pair. Moreover,   for $T \in \tilde{\T}$ the glueing sequence is  $\, \, 0 \rightarrow \tau_\id T \rightarrow T \rightarrow  T/\tau_\id T \rightarrow 0$.
\end{itemize}
\end{pro}
\begin{proof}
Since $\mathbb I_\infty \subseteq \mathbb I_0$ then condition T1) in the definition of torsion pair holds.

Assume now that $\id$ is a strong idempotent ideal. Then $\modu \LA \subset \T$. In fact, if $X \in \modu \LA$ then $\Ext^{i}_{\Lambda}(\LA ,X)\cong \Ext^{i}_{\LA}(\LA ,X)=0$ for all $i \geq 0$, so $X \in \T$. On the other hand, we know that a module $Y$ is in $\mathbb I_{\infty}$ if and only if $\Ext^{i}_{\Lambda}(\LA ,Y)=0$ for all $i \geq 0$, by \cite{APT}, Proposition 2.6. Thus, it follows from the definition of $\T$ that  $\mathbb I_{\infty}=\mathbb I_{0} \cap \T$.

Therefore, for $T \in \T$ the exact sequence $0 \rightarrow \tau_{\LA} T \rightarrow T \rightarrow  T/\tau_{\LA} T \rightarrow 0$ has $\tau_{\LA} T$ in $\modu \LA$ and $T/\tau_{\LA} T$ in $\mathbb I_{\infty}$ and is then a glueing sequence for $T$. This proves condition T2) and ends the proof of (a).  The proof of (b) is similar.

\end{proof}

\begin{rk} 
\IP {Though we do not know} wether the subcategories $\T$ and $\tilde \T$ have kernels and cokernels, we observe that the category $\T$ is not in general abelian as the following simple example shows. Let $\Lambda$ be the path algebra of the quiver $ 1 \rightarrow 2$, and $P = S_2$, the simple projective module associated to the vertex $2$.  Then, $\id =\tau_P \Lambda \simeq  S_2\oplus S_2$ is a projective $\Lambda $ module and therefore it is a strong idempotent ideal.  Moreover,  $\T =\add  \,\{\begin{smallmatrix}\ S_1\ \\S_2\end{smallmatrix},{\begin{smallmatrix}\ S_1\end{smallmatrix}}\}$ and $\LA \simeq S_1$. Consider the exact sequence $0 \rightarrow S_2 \rightarrow    \begin{smallmatrix}\ S_1\ \\S_2\end{smallmatrix} \xrightarrow{f} S_1 \rightarrow 0$. Then   $f$ is a map in $\T$, and $\Ker_\T (f) =0 $, because there are no nonzero maps from objects in $\T$ to $S_2$. Thus \IP {$\Coker_{\T}(\Ker_{\T} (f) ) = \Coker_{\T} (  0 \rightarrow    \begin{smallmatrix}\ S_1\ \\S_2\end{smallmatrix})  = ( \begin{smallmatrix}\ S_1\ \\S_2\end{smallmatrix} \xrightarrow{id}    \begin{smallmatrix}\ S_1\ \\S_2\end{smallmatrix})$. However, ${\Ker}_{\T}({\Coker}_{\T} (f) ) = {\Ker}_{\T} (S_1 \rightarrow    0)= (S_1\xrightarrow{id} S_1)$.}
\end{rk}

\vskip .25in

In connection with the torsion pairs and subcategories above considered we prove two technical lemmas which will be useful throughout the paper.

\begin{lem} Let $\id$ be an idempotent ideal. Then

\begin{itemize}

\item [(a)]  $\Hom_{\Lambda}(P,X) \simeq \Hom_{\Lambda}(P, X/\tla X)$ for  all $X \in \modu \Lambda$.

\item [(b)] $\Ext_{\Lambda}^j(-, X) |_{\Pinf}\simeq \Ext_{\Lambda}^j(-,X/\tla X)|_{\Pinf}$ for all $j \geq 0$ and $X \in \modu\Lambda$.  

\item [(c)] $\Hom_{\Lambda}(X,I) \simeq \Hom_{\Lambda}(\ta X, I)$ for  all $X \in \modu \Lambda$.

\item [(d)] $\Ext_{\Lambda}^j(X,-) |_{\I_\infty}\simeq \Ext_{\Lambda}^j(\tau_\id X,-)|_{\I_\infty}$ for all $j \geq 0$ and $X \in \modu\Lambda$.  

\end{itemize}

\begin{proof}

(a) The result follows directly by applying the exact functor $\Hom_\Lambda (P,-)$ to the exact sequence $0 \rightarrow \tla X \rightarrow X \rightarrow X/\tla X \rightarrow  0$.

(b) We recall from \cite{APT}, Theorem 3.2 c), that there is a functorial isomorphism 
$\Ext_{\Lambda}^j(-, Y) |_{\Pinf}\simeq \Ext_{\Gamma}^j((P,-), (P,Y)) |_{\Pinf}$ for all $Y$ in mod$\Lambda$. The result follows now using (a).

By duality we obtain the statements (c) and (d). 

\end{proof}

\end{lem}

When we further assume that the ideal $\id$  is  strong idempotent we get the following result.

\begin{lem} Let $\id$ be a strong idempotent ideal. Then

\begin{itemize}

\item [(a)] $\Ext_{\Lambda}^j(-, X/\tla X )|_{\modu \Lambda/ \id }=0$ for all  $X \in \T$ and $j \geq 1$.

\item[(b)] $\Ext^j_{\Lambda}(-,X_1) \simeq \Ext^j_{\Lambda}(-,X_2)$ with $j\geq 0$ implies 

\noindent $\Ext^j_{\Lambda} (-, X_1/\tla X_1)|_{\tilde{\T}} \simeq 
\Ext^j_{\Lambda} (-, X_2/\tla X_2)|_{\tilde{\T}}$ for all $X_1, X_2 \in \T$.

\item [(c)] $\Ext_{\Lambda}^j(\tau_\id X, -)=0$ for all  $X \in \tilde \T$ and $j \geq 1$.

\item[(d)] $\Ext^j_{\Lambda}(X_1,-) \simeq \Ext^j_{\Lambda}(X_2,-)$  with $j\geq 0$ implies 

\noindent $\Ext^j_{\Lambda} (\tau_\id X_1,-)|_{{\T}} \simeq 
\Ext^j_{\Lambda} (\tau_\id X_2,-)|_{{\T}}$ for all $X_1, X_2 \in \tilde \T$.

\end{itemize}

\begin{proof}

(a) Let $X$ in $\T$ and $Z$ in mod$\LA$. Then $X/\tla X \in \Iinf$ and using (d) of the previous lemma we conclude that $\Ext_{\Lambda}^j(Z,X/\tla X)\simeq \Ext_{\Lambda}^j(\tau_\id Z,X/\tla X)=0$, since $\tau_\id Z =0$ because $Z$ is a $\LA$-module.

(b) Let $X_1, X_2  \in \T  $ be such that $\Ext^j_{\Lambda}(-,X_1) \simeq \Ext^j_{\Lambda}(-,X_2)$. So  $L_1= X_1/\tla X_1, L_2= X_2/\tla X_2 \in  \Iinf$.  Then we obtain from (a)  that
$\Ext_{\Lambda}^j(-, L_i)|_{\modu \Lambda/ \id }=0$ for   $ i=1,2$ and for all $j \geq 1$.

\noindent Let $Z \in \tilde{\T}$.  Applying the functor $\Hom_{\Lambda}(-, L_i)$ to the glueing sequence
\[\xymatrix{    0 \ar[r] &  \ta Z \ar[r] & Z \ar[r]  & Z/\ta Z \ar[r] & 0  
}\]

\noindent the corresponding long exact sequence yields isomorphisms
$$\Ext^j_\Lambda (Z,L_i) {\simeq}\Ext^j_\Lambda (\ta Z, L_i), $$ \noindent for $i= 1, 2 $ and $j\geq 1$.

On the other hand, by \IP{(b)} of the previous lemma we know that  $\Ext_{\Lambda}^j(\ta Z, X_i) \simeq \Ext_{\Lambda}^j(\ta Z,L_i)$ for $i=1,2, j\geq 1$ \IP{because $\id$ is a strong idempotent ideal, so $Z$ in $\tilde{\T}$ implies that $Z$ is a $\LA$-module. }
Then, in the commutative diagram

\[\xymatrix{   \Ext^j (Z,L_1) \ar[r] \ar[d] & \Ext^j (\ta Z, L_1)   \ar[d] &
 \\
\Ext^j (Z,L_2) \ar[r]  & \Ext^j (\ta Z, L_2)  &
}\]

\noindent the horizontal arrows and the right vertical arrow are isomorphisms when $j\geq 1$. This proves  the left vertical arrow is also an isomorphism in this case, as desired. 

Finally, statements (c) and (d)  follow from (a) and (b) by duality.

\end{proof}
\end{lem}

\section{Main results}

Next we turn our attention to the functions $\phi$ and $\psi$ defined by Igusa and Todorov. We are going to use the characterization of these functions in terms of the functor Ext given in \cite{FLM}. We start with two lemmas comparing the behaviour of this functor  in mod$\Lambda$ and in mod$\LA$.

\begin{lem}\label{Lema1Jue1}
Let $\id$ be an idempotent ideal such that $\pd(_{\Lambda}\LA)=r<\infty$. Let $X_1, X_2\in\modu \LA$ and $t \geq 1$. Then  $\Ext^t_{\LA}(X_1, -)\simeq \Ext^t_{\LA}(X_2, -)$ implies
$\Ext^{t+r}_{\Lambda}(X_1, -)\simeq \Ext^{t+r}_{\Lambda}(X_2, -)$.
\end{lem}
\begin{proof}
Note first that $\Ext^{t+r}_{\Lambda}(X_j, -)\simeq \Ext^{t}_{\Lambda}(X_j,\Omega^{-r}(-))$, for $j=1,2$. 
Since $\Omega^{-r}(\modu\Lambda)\subset \T$, from Lemma 5.5, \cite{APT}, using Proposition 1.1, \cite{APT}, it follows that $\Ext^{i}_{\LA}(X_j,\tla(\Omega^{-r}(-)))\simeq  \Ext^{i}_{\Lambda}(X_j,\Omega^{-r}(-))$, for all $1\leq i$, and $j=1,2$.

Now $\Ext^{t+r}_{\Lambda}(X_1,-)\simeq \Ext^{t}_{\Lambda}(X_1,\Omega^{-r}(-))\simeq \Ext^t_{\LA}(X_1,\tla(\Omega^{-r}(-)))\simeq \Ext^t_{\LA}(X_2,\tla(\Omega^{-r}(-)))\simeq \Ext^{t+r}_{\Lambda}(X_2,-)$ \end{proof}

\begin{lem}\label{Lema2Jue3}
Let $\id$ be an idempotent ideal such that $\pd(\LA_{\Lambda})=r<\infty$. Let $X_1, X_2\in \modu \LA$ and $t \geq 1$.
Then $\, \, \Ext^{t}_{\LA}(-,X_1)\simeq \Ext^{t}_{\LA}(-,X_2)\, \, $  implies $\Ext^{t+r}_{\Lambda}(-,X_1)\simeq \Ext^{t+r}_{\Lambda}(-,X_2)$.
\end{lem}
\begin{proof}
It follows from Lemma \ref{Lema1Jue1} by duality, using that $\Ext^j_{\LA}(-,Y)\simeq \Ext^j_ {\LA ^{op}}(DY,D(-))$.
\end{proof}

We prove next that when the ideal $\id$ is a strong idempotent ideal of finite projective dimension then the $\phi$ dimension of the factor algebra $\LA$ is bounded by the $\phi$ dimension of $\Lambda$.

\begin{teo}\label{desigualdades}
Let $\id$ be a strong idempotent ideal of $\Lambda$ such that $\pd(_{\Lambda}\LA)=r<\infty$. Then 
\begin{itemize}
\item[(a)] $\phi^{\LA}_l(X)\leq \phi^{\Lambda}_l(X) \, \,$  for all $ \, \, X \in \modu \LA$.

\item [(b)] $\phd\LA\leq\phd\Lambda$.
\end{itemize}
\end{teo}
\begin{proof} Let $d$ a positive integer and assume that $X=X_1\oplus X_2$ is a $d$-division of the $\LA$-module $X$. 
This is, $\Ext_{\LA}^d(X_1, -)\not\simeq\Ext_{\LA}^d(X_2, -)$ and $\Ext_{\LA}^{d+1}(X_1, -)\simeq\Ext_{\LA}^{d+1}(X_2, -)$.
Since $\id$ is a strong idempotent ideal then $\Ext_{\Lambda}^d(X_1, -)\not\simeq\Ext_{\Lambda}^d(X_2, -)$. On the other hand, we know that $\phi_l ^\Lambda(X) = max(\{ d \in \N : $ there is a $d$-division of $ X $ in $\modu \Lambda  \}\cup\{ 0\})$, by Theorem 3.6 in \cite{FLM}. Thus to prove that $\phi^{\Lambda}_l(X) \geq d$, it is enough to find $l> d$ such that   $\Ext_{\Lambda}^l(X_1, -)\simeq\Ext_{\Lambda}^l(X_2, -)$. In fact, if $l_0 $ is minimal with this property, then $X=X_1\oplus X_2$ is an $(l_0-1)$-division of the $\Lambda$-module $X$.

Since we assumed  that $\Ext_{\LA}^{d+1}(X_1, -)\simeq\Ext_{\LA}^{d+1}(X_2, -)$ then, using Lemma 4.1,  we obtain that  $\Ext_{\Lambda}^{d+r+1}(X_1, -)\simeq\Ext_{\Lambda}^{d+r+1}(X_2, -)$.

Hence, $l= d+r+1 > d$ satisfies $\Ext_{\Lambda}^l(X_1, -)\simeq\Ext_{\Lambda}^l(X_2, -)$. So we found $l$ as required, proving that $\phi^{\Lambda}_l(X) \geq d$.  This proves (a), and (b) follows immediately. 
\end{proof}

\IP{We observe that $\phi_l$ can be replaced by $\phi_r$ in the previous theorem, since $\phi_r(X) = \phi_l(DX)$.
\vskip.2in

These results apply to any convex subcategory $\Delta$  of a quiver algebra $\Lambda = kQ/I$, where $Q$ is a finite quiver and $I$ is an admissible ideal of the path algebra $kQ$.   That is, $\Delta = kQ'/(kQ'\cap I)$, where $Q'$ is a full convex subquiver of $Q$. In this case $\Delta = \LA$, where $\id$ is the trace of the projective module $P= \bigoplus_{i \notin Q_0'} P_i$. In this situation it is known that $\id$ is a strong idempotent ideal (\cite{M}, Ch. II, Lemma 3.7) and we obtain the following corollary.

\begin{cor}
Let $\Delta$ be a full convex subcategory of the quiver algebra  $\Lambda$, and let $\id$ be the idempotent ideal such that $\Delta = \LA$. If $\id$ has finite projective dimension then  $\phd\Delta\leq\phd\Lambda$.

\end{cor}

Now we turn our attention to $\Gamma$-modules. We study the behaviour of both Igusa-Todorov functions $\phi$ and $\psi$ under the functor $\Hom_\Lambda(P, -): \rm{mod}\Lambda \rightarrow \rm{mod}\Gamma$ restricted to the subcategories  $\Pinf $ and $\Iinf$ of $ \modu \Lambda$.}

\begin{pro}\label{igualdad} For a $\Lambda$-module $Y\ \in \Pinf$ the following properties hold:

\begin{itemize}
\item[(a)]
 there exists a $d$-division of $Y$ if and only if there exists a $d$-division of $\Hom_{\Lambda}(P,Y)$.
\item[(b)]
 $\phi^{\Lambda}_l(Y)=\phi_l^{\Gamma}(\Hom_{\Lambda}(P,Y)).$
\item [(c)]
 $\psi_l^{\Lambda}(Y)=\psi_l^{\Gamma}(\Hom_{\Lambda}(P,Y)).$
\end{itemize}
\end{pro}
\begin{proof}

(a)
Since $\Hom_\Lambda(P, -): \rm{mod}\Lambda \rightarrow \rm{mod}\Gamma$ induces an equivalence of categories $\Puno \rightarrow \rm{mod}\Gamma$ and $Y\in \Pinf \subseteq\Puno $,   it follows that $Y=Y_1\oplus Y_2$ if and only if $\Hom_{\Lambda}(P,Y)=\Hom_{\Lambda}(P,Y_1)\oplus\Hom_{\Lambda}(P,Y_2)$. The statement follows from the fact that $Y_1, Y_2\in \Pinf$ implies $\Ext^i_{\Lambda}(Y_j,-)\simeq \Ext^i_{\Gamma}(\Hom_{\Lambda}(P,Y_j), \Hom_{\Lambda}(P,-))$, for $j=1,2$ and for all $i \geq 0$ (\cite{APT},Theorem 3.2,(c)).
\vskip .05in

(b)
This is a direct consequence of (a), using that $\phi^{\Lambda}_l(Y)=n$ if and only if $n=$max$(\{\ d\in\N:$ there exists a $d$-division of $Y\ \}\cup {0})$ by Theorem 3.6 in \cite{FLM}.
\vskip .05in

(c)
Let $X\in $ mod$\Gamma$. Then $\psi_l^\Gamma(X) = n+l$, with $n= \phi_l^\Gamma(X)$ and $l=\pd Z_1,$ where $Z_1$ is the largest summand of $\Omega^n(X)$ of finite projective dimension. We write $\Omega^n(X)= Z_1 \oplus Z_2$.

Let $Y\in  \Pinf \subseteq\Puno$ be such that $X = \Hom_\Lambda(P,Y)$. Since $\phi_l^\Lambda (Y) = \phi_l^\Gamma(X)$, we only need to prove that $l$ is the projective dimension of largest summand $Y_1$ of $\Omega^n(Y)$  of finite projective dimension. Let $\Omega^n(Y) = Y_1 \oplus Y_2$ and let $$ \cdots P_{n+1} \rightarrow P_n \rightarrow \cdots \rightarrow P_0 \rightarrow Y \rightarrow 0$$ be a minimal projective resolution of $Y$ in mod$\Lambda$. Since $Y\in \Pinf $ then
 $$ \cdots \rightarrow \Hom_\Lambda(P,P_n) \rightarrow \cdots \rightarrow \Hom_\Lambda(P,P_0) \rightarrow \Hom_\Lambda(P,Y) \rightarrow 0$$ 
\noindent is a minimal projective resolution of $X =\Hom_\Lambda(P,Y)$ in mod$\Gamma$, \MIP {as follows from 
 \cite{APT}, Lemma 3.1.}

Therefore $\Omega^n(X)\simeq  \Hom_\Lambda(P,\Omega^n(Y))$. This is, $Z_1 \oplus Z_2\simeq\Hom_\Lambda(P,Y_1) \oplus \Hom_\Lambda(P,Y_2)$. Since $Y\in \Pinf$, then $\Omega^n(Y)$  and   all  direct summands of $\Omega^n (Y)$ are also in $\Pinf$. Thus $\pd L = \pd \Hom_\Lambda (P,L)$ for any summand L of $\Omega^n(Y)$ (see \cite{APT}, Corollary 3.3). From this we conclude that the projective dimensions of the largest summands of finite projective dimension of $\Omega^n(X)$ and $\Omega^n(Y) $ coincide, as desired.\end{proof}

We state the corresponding result for the $\phi$-injective dimension in the next proposition.

\begin{pro} For a  $\Lambda$-module $Y\ \in \Iinf$ the following properties hold:

\begin{itemize}
\item[(a)]
 there exists a $d$-division of $Y$ if and only if there exists a $d$-division of $\Hom_{\Lambda}(P,Y)$.
\item[(b)]
 $\phi_r^{\Lambda}(Y)=\phi_r^{\Gamma}(\Hom_{\Lambda}(P,Y)).$
\item[(c)]
 $\psi_r^{\Lambda}(Y)=\psi_r^{\Gamma}(\Hom_{\Lambda}(P,Y)).$
\end{itemize}
\end{pro}
\begin{proof}
The result follows using that $\Hom_{\Lambda}(P,-): \Iuno \rightarrow \rm{mod}\Gamma$ is an equivalence of categories, the fact that  $\Ext^i_{\Lambda}(-,Y)\simeq \Ext^i_{\Gamma}(\Hom_{\Lambda}(P,-), \Hom_{\Lambda}(P,Y))$ for all $Y \in \Iinf$ and $i \geq 0$  (\cite{APT}, Theorem 3.2, (b)),
 and dualizing arguments in the proof of the previous \MIP {proposition}.\end{proof}

Since $\Hom_\Lambda(P, -): \rm{mod}\Lambda \rightarrow \rm{mod}\Gamma$ induces  equivalences of categories $\Puno \rightarrow \rm{mod}\Gamma$ and $\Iuno \rightarrow \rm{mod}\Gamma$, the previous \MIP {propositions} yield the following result.

\begin{teo}\label{puno=pinfinito}
\begin{itemize}
\item[(a)] If $\Puno=\Pinf$ then $\phd\Gamma\leq\phd\Lambda$ and ${\psi_l} {\rm dim}\Gamma \leq {\psi_l} {\rm dim} \Lambda$.

\item[(b)] If $\Iuno=\Iinf$ then $\phinjd \Gamma\leq\phinjd\Lambda$ and $\psi_l {\rm dim}\Gamma\leq\psi_l {\rm dim}\Lambda$.
\end{itemize}
\end{teo}

\IP{Our next objective is to find bounds for the $\phi$ dimension of $\Lambda$ in terms of the $\phi$ dimensions of $\LA$ and $\Gamma$.

We start with the case when the global dimension of $\Gamma $ is finite. Let $\T$ be the subcategory  of mod$\Lambda$  considered in the previous section, consisting of the modules $T$ such that $\Ext^{i}_{\Lambda}(\LA ,T)=0$ for  all $i\geq 1$. 
We observe first that  bounds for the function $\phi$ in the subcategory $\T$ will give us  bounds for $\phi$ in mod $\Lambda$, as the following lemma shows.

\begin{lem}\label{otro} Let $\id$ be an idempotent ideal. Then 
\begin{itemize}
\item[(a)] $\phinjd (\Lambda)\leq \pd_{\Lambda}(\LA) + \phinjd(\T)$

\item[(b)] $\phd (\Lambda)\leq \pd(\LA)_{\Lambda} + \phd(\tilde \T)$.
\end{itemize}
\end{lem}

\begin{proof}If $\pd_{\Lambda}(\LA)=\infty$ there is nothing to prove. Assume $\pd_{\Lambda}(\LA)=t<\infty$ and let $X \in \modu \Lambda$.
Then by  Lemma 5.5 in \cite{APT}, $\Omega^{-t}(X)\in \T$. The lemma follows \MIP{now} by repeated use of the dual of the inequality of 
Lemma 1.3 in \cite{HLM1}, $\phi_r(X)\leq t+ \phi_r(\Omega^{-t}(X)$
. This proves (a), and (b) follows by duality.
\end{proof}

Observe now that \MIP{when} $\id$ is a strong idempotent ideal then $\id$ is in $\Pinf$ (\cite{APT}, Theorem 2.1'), so $\pd_{\Lambda}(\id)\leq \pd_{\Gamma}(\Hom_{\Lambda}(P,\id))\leq \gld \Gamma $. Thus 
$\pd_{\Lambda}(\LA)\leq  \gld \Gamma +1$. Since being a strong idempotent ideal is a symmetric condition we obtain that
$\pd (\LA)_{\Lambda} \leq  \gld \Gamma +1$, as observed in \cite{APT} at the end of section 5. 
 }

\begin{pro}\label{T}
Let $\id$ be a strong idempotent ideal.  
Then 
$$\phinjd(\T)\leq \rm{max}\{\ \gld(\Gamma)+1,\ \phinjd(\LA)+ \pd(\LA)_{\Lambda}\ \}. $$
\end{pro}
\begin{proof}

Let now $r= \pd(\LA)_{\Lambda}$,  $T$ in $\T $ and consider the glueing sequence  $0 \rightarrow \tla T \rightarrow T \rightarrow T/\tla T \rightarrow  0$  in $\T$. Since $T/\tla T$ is in $\Iinf$, we know by \cite{APT}, Corollary 3.3 b) that  $\injd _{\Lambda } T/\tla T= \injd _{\Gamma} \Hom_\Lambda (P,T/\tla T)$. Then,  the corresponding long exact sequence of functors yields isomorphisms of functors $\delta_i: \Ext^i_{\Lambda}(-,\tla T) \rightarrow \Ext^i_{\Lambda}(-,T)$ for $i > \gld \Gamma +1$.

It is enough to show that $\phi_r^{\Lambda}(T)\leq  \phi_r^{\LA}(\tla T )+r$, whenever 
$\phi_r^{\Lambda}(T) > \gld \Gamma+1$. With this purpose we assume that $d=\phi_r^{\Lambda}(T) > \gld \Gamma+1$ and that 
$T=T_1\oplus T_2$ is a $d$-injective-division in $\modu\Lambda$. We start by proving that $\tla T = \tla T_1 \oplus \tla T_2$ is a j-injective division in mod$\Lambda/\id$, for some $j$ such that $d+1\leq j+r$. 
Since $T=T_1\oplus T_2$ is a $d$-injective-division of $T$,  then $\Ext^d_{\Lambda}(-,T_1)\not\simeq \Ext^d_{\Lambda}(-,T_2)$ and $ \Ext^{d+1}_{\Lambda}(-,T_1)\simeq \Ext^{d+1}_{\Lambda}(-,T_2).$ Therefore, since $T_1$ and $T_2$ are also in $\T$ we have that $\Ext^d_{\Lambda}(-,\tla T_1)\not\simeq \Ext^d_{\Lambda}(-,\tla T_2)$  and $\Ext^{d+1}_{\Lambda}(-,\tla T_1)\simeq \Ext^{d+1}_{\Lambda}(-,\tla T_2)$ in $\modu \Lambda$, by the isomorphisms above.

Now, since $\id$ is a strong idempotent ideal we deduce from the last isomorphism that  
$\Ext^{d+1}_{\LA}(-,\tla T_1)\simeq \Ext^{d+1}_{\LA}(-,\tla T_2)$  in mod$\Lambda /\id$.  

On the other hand,   $\pd(\LA)_{\Lambda}\leq  \gld \Gamma +1$ as we observed just before the statement of the proposition. Since we assumed that $\gld(\Gamma)+1 < d$ we obtain that  $\pd(\LA)_{\Lambda} + 1=r+1 \leq d$.                       
 We conclude then, from Lemma \ref{Lema2Jue3} that $\Ext^{d-r}_{\LA}(-,\tla T_1)\not\simeq \Ext^{d-r}_{\LA}(-,\tla T_2)$ in mod$\Lambda/\id$. 
This fact and the isomorphism $\Ext^{d+1}_{\LA}(-,\tla T_1) \simeq \Ext^{d+1}_{\LA}(-,\tla T_2)$ obtained above imply that $\tla T = \tla T_1 \oplus \tla T_2$ is a j-injective division in mod$\Lambda/\id$, for some $j$ such that $d-r\leq j\leq d$. 
Thus $d \leq j+r \leq \phinjd \LA+r$. This proves  that $\phi_r(T)\leq  \phinjd \LA + r$, provided 
$\gld \Gamma + 1 > d  = \phi_r(T)$, and ends the proof of the proposition. \end{proof}

\begin{pro}
Let $\id$ be a strong idempotent ideal. Then $$\phinjd (\Lambda)\leq \pd_{\Lambda}(\LA) +{\rm max}\{\ \gld(\Gamma)+1, \ \pd(\LA)_{\Lambda}+\phinjd(\LA)\ \}.$$
\end{pro}

\begin{proof}
The result follows from  Lemma 4.8 and Proposition
 \ref{T}.
\end{proof}

\begin{cor}
Let $\id$ be a strong idempotent ideal. Assume $\gld(\Gamma)< \infty$, then $\phinjd(\LA)$ is finite if and only if 
$\phinjd (\Lambda)$ is finite.
\end{cor}
\IP{Since convex subalgebras of $\Lambda$ \MIP{are obtained as} factors of $\Lambda $ by a strong idempotent ideal, the previous results apply to them. In particular, we obtain the following corollary.
\begin{cor}
Let $\Delta$ be a full convex subalgebra of $\Lambda$, and let $\id= \tau _P(\Lambda)$ be the idempotent ideal such that $\Delta = \LA$. If $\Gamma$ has finite global dimension, then  $\phinjd (\Delta)$ is finite if and only if 
$\phinjd (\Lambda)$ is finite.

\end{cor}

\begin{ex}
Let $\Lambda =
\left(\begin{matrix}A&0\\M&B\end{matrix}\right)$, where $A$ and $B$ are artin algebras, $\gld(B)< \infty$  and $M$ is a $B$-$A$-bimodule. Then $\phinjd (\Lambda)$ is finite if and only if $\phinjd (A)$ is finite.
\end{ex}

In particular we obtain that a one point co-extension of $A$ has finite $\phinjd$ if and only if $A$ does.

\vskip .2in

Next we illustrate the previous Proposition with the following example.

\begin{ex}
Let $\Lambda$ be the algebra given by the quiver $Q$ 
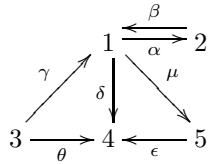
\[\xymatrix{ & 1 \ar@<2pt>[r]_{\alpha}\ar@<2pt>[d]_{\delta}\ar@<2pt>[dr]^{\mu}& 2\ar@<-6pt>[l]_{\beta} \\ 
3 \ar@<2pt>[ur]^{\gamma}\ar[r]_{\theta}&4 & 5 \ar[l]^{\epsilon}  }\]
with relations $\alpha \beta = \beta \alpha =0, \mu \gamma =0, \delta \gamma =0, \epsilon \mu = 0$.  Let $P = P_3\oplus P_4\oplus P_5$, and let $\id = \tau_P(\Lambda)$. Then $\id \simeq P_3\oplus P_4^3\oplus P_5\oplus S_5^2 $, so
$\pd \id =1$. Moreover, since there is an exact sequence $0\rightarrow P_4 \rightarrow  P_5 \rightarrow  S_5\rightarrow  0$ we obtain that $S_5 \in \Pinf$, so $\id \in \Pinf$ and is thus a strong idempotent. Then the quiver of $\LA$ is 
 
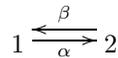
\[\xymatrix{  1 \ar@<2pt>[r]_{\alpha}& 2\ar@<-6pt>[l]_{\beta}   }\]
with radical square zero, so $\LA$ is selfinjective and therefore $\phinjd \LA = 0$, by \cite{HL}. On the other hand $\Gamma $ is the hereditary algebra with quiver   \[\xymatrix{ 3 \ar[r]_{\theta}&4 & 5 \ar[l]^{\epsilon}  }\] and we get

 $$\phinjd (\Lambda)\leq 2 +{\rm max}\{\ 1+1, \ 2+1\ \}=5.$$

\end{ex}

}

We now turn our attention to the case when $\gld {(\LA)}$ is finite.

\begin{pro}
Let $\id$ be a strong idempotent ideal and assume that $\gld(\LA)$ is finite. Then 
\begin{itemize}

\item[(a)]$\phi_r(T)\leq \rm{max}\{\ \gld(\LA)+  \pd(\LA)_{\Lambda} +1,\ \phinjd(T/\tla T)\ \}$ for any $T \in \modu\Lambda$.

\item[(b)] $\phinjd (\T)\leq $ max$\{\ \gld(\LA)+  \pd(\LA)_{\Lambda} +1,\ \phinjd(\Gamma)\ \}$.

\end{itemize}

\end{pro}

\begin{proof}

Let $s=\gld(\LA), \, r=  \pd(\LA)_{\Lambda}$ and let $X$ be a $\LA$-module.

We claim that $\Ext^i_{\Lambda}(-,X)=0$ for all $i \geq s+r+1$. In fact, since $\gld\LA=s$ we know that $\Ext^j_{\LA}(-,X)=0$ for all $j \geq s+1$. Then $\Ext^{j+r}_{\Lambda}(-,X)=0$ by Lemma 4.2, for $j \geq s+1$. So the claim holds.

Let now $T$ in $\modu \Lambda $ and consider the sequence  $0 \rightarrow \tla T \rightarrow T \rightarrow T/\tla T \rightarrow  0$. The corresponding long exact sequence of functors yields an isomorphism of functors $\delta_i: \Ext^i_{\Lambda}(-,T) \rightarrow \Ext^i_{\Lambda}(-,T/\tla T)$ for each $i \geq s+r+1$.

We prove next that a $d$-division of $T$ in $\modu \Lambda$ yields a $d$-division $T/\tla T$ in $\modu \Lambda$. 

In fact, let  $d=\phinjd (T)$ and let $T=T_1 \oplus T_2$ be a $d$-division of $T$. This is, $\Ext^d_{\Lambda}(-,T_1) \not \simeq  \Ext^d_{\Lambda}(-,T_2)$ and $\Ext^{d+1}_{\Lambda}(-,T_1)  \simeq  \Ext^{d+1}_{\Lambda}(-,T_2)$.
Assume now $d \geq s+r+1$. The functorial isomorphisms $\Ext^{i}_{\Lambda}(-,T_k) \simeq  \Ext^i_{\Lambda}(-,T_k/\tla T_k)$, $k=1,2$ and $i \geq s+r+1$ induced by $\delta_i$  show that $T/\tla T= T_1/\tla T_1 \oplus T_2/\tla T_2$ is a $d$-division of $T/\tla T$ in $\modu \Lambda$. Therefore, $\phi^{\Lambda}_l(T/\tla T) \geq d$. This ends the proof of $(a)$.

Assume finally that $ T \in \T$. Therefore $T/\tla T \in \I_{\infty}$. We know by  Proposition 4.6 (b) that $\phi_r ^\Gamma \Hom_\Lambda(P,T/\tla T) \geq d$ . This proves that $\phinjd \T \leq $ max $\{ s+r+1, \phinjd \Gamma \}$.
\end{proof}

\begin{pro}
Let $\id$ be a strong idempotent ideal. Then
$\phinjd (\Lambda)\leq  \pd _{\Lambda}\LA$  + $\max\{\ \gld(\LA)+  \pd(\LA)_{\Lambda} +1,\ \phinjd(\Gamma)\ \}$
\end{pro}

\begin{proof} If $\gld\LA=\infty$ there is nothing to prove. If $\gld\LA$ is finite, the result follows from the previous proposition and Lemma \ref{otro} \end{proof}


\IP {Next we obtain  another bound for $\phd (\Lambda)$ with different methods.
}

\begin{lem}
Let $\id$ be a strong idempotent ideal. Let 
$0 \rightarrow X \rightarrow Y \rightarrow Z \rightarrow 0$ be an exact sequence in $\modu \Lambda$ with $X \in \Pinf$ and $Z \in \modu \LA$ such that $\pd _{\LA}Z$ is finite. Then 
\begin{itemize}
\item[(a)] $\Omega^\IP{n}(Y) \in \Pinf $ for $\IP{n} \geq \pd _{\LA}Z$. 

\item[(b)] $\phi_l^\Lambda(Y) \leq \pd _{\LA}Z +  \phd \Gamma $. 
\end{itemize}
\end{lem}

\begin{proof}
(a) Let $ \cdots P_n \rightarrow   \cdots \rightarrow P_2 \rightarrow P_1 \rightarrow P_0 \rightarrow X \rightarrow 0$
and  $ \cdots Q_n \rightarrow   \cdots \rightarrow Q_2 \rightarrow Q_1 \rightarrow Q_0 \rightarrow Z \rightarrow 0$ be minimal projective resolutions in $\modu \Lambda$.

Since $\id$ is a strong idempotent ideal then $ \cdots Q_n/ \id Q_n \rightarrow   \cdots \rightarrow Q_2/ \id Q_2 \rightarrow Q_1/ \id Q_1 \rightarrow Q_0/ \id Q_0 \rightarrow Z \rightarrow 0$ is a minimal projective resolution of $Z$ in $\modu \LA$ (Theorem 1.6 iii) in \cite{APT}). Then $Q_n/ \id Q_n=0$ for $n > s=\pd _{\LA}Z$. So $Q_n=\id Q_n=\tau_P Q_n$ is in $\add P$ for $n >s$ and therefore $\Omega^n(Z) \in \Pinf$  for $n >s$.

On the other hand, we assumed that $X$ is in $\Pinf$, so that $\Omega^n(X) $ is also in $\Pinf$.
Let $n >s$. Since $ \cdots P_r\oplus Q_r \rightarrow  \cdots \rightarrow P_1\oplus Q_1 \rightarrow P_0\oplus Q_0 \rightarrow Y \rightarrow 0$ is a projective resolution of $Y$ and $P_r\oplus Q_r\in \add P$ for $r\geq n$, then $\Omega^n(Y)\in \Pinf$.
This proves (a).

(b) Let $s= \pd _{\LA}Z$. By a) we know that $\Omega^n(Y)$ is in $\Pinf$, for all $n\geq s$. Thus  
$\phi^{\Lambda}_l(Y)\leq \phi^{\Lambda}_l(\Omega^s(Y)) +s = \phi_l^{\Gamma}(\Hom_{\Lambda}(P,\Omega^s(Y))) + s\leq \phd\Gamma +s$, where the first inequality is given by Lemma 1.3 in \cite{HLM1}, and the equality follows from \IP {Proposition \ref{igualdad} (b).}
\end{proof}

\begin{pro}
Let $\id$ be a strong idempotent ideal. Then 
\begin{itemize}
\item[(a)]
$\phd(\tilde{\T})\leq\gld(\LA)+\phd\Gamma.$
\item[(b)]
$\phd (\Lambda)\leq  \pd(\LA)_{\Lambda} + \gld(\LA)+ \phd(\Gamma)$.
\end{itemize}
\end{pro}
\begin{proof}

(a) Taking supremum on $T\in \tilde{\T}$ and using (b) of the previous lemma applied to the glueing sequence   $\, \, 0 \rightarrow \tau_\id T \rightarrow T \rightarrow  T/\tau_\id T \rightarrow 0$, we get $\phd(\tilde{\T})\leq\phd\Gamma+\gld(\LA).$

(b) If $\gld\LA=\infty$ there is nothing to prove. If $\gld\LA$ is finite, by \IP{Lemma \ref{otro} (b)}, we get
that $\phd (\Lambda)\leq  \pd(\LA)_{\Lambda} + \phd(\tilde{\T}) \leq  \pd(\LA)_{\Lambda} + \gld(\LA)+ \phd(\Gamma)$.

\end{proof}

\footnotesize

\vskip3mm \noindent Mar\'\i a Andrea Gatica:\\
Instituto de Matem\'atica de Bah\'\i a Blanca,\\
Universidad Nacional del Sur,\\
Av. Alem 1253, B8000CPB,\\
Bah\'\i a Blanca, ARGENTINA.

{\tt mariaandrea.gatica@gmail.com}

\vskip3mm \noindent Marcelo Lanzilotta:\\
Instituto de Matem\'atica y Estad\'\i stica Rafael Laguardia (IMERL),\\
Universidad de la Rep\'ublica.\\ 
J. Herrera y Reissig 565 C.P. 11300, Montevideo, URUGUAY.

{\tt marclan@fing.edu.uy}

\vskip3mm \noindent Mar\'\i a In\'es Platzeck:\\
Instituto de Matem\'atica de Bah\'\i a Blanca,\\
Universidad Nacional del Sur,\\
Av. Alem 1253, B8000CPB,\\
Bah\'\i a Blanca, ARGENTINA.

{\tt platzeck@uns.edu.ar}


\begin{thebibliography}{60}
\bibitem{APT} M. Auslander, M. I. Platzeck, G. Todorov. Homological theory of idempotent ideals. {\it Trans Am. Math. Soc.}, 
vol. 332, n. 2, 667-692, (1992).

\bibitem{BR} A. Beligiannis, I. Reiten. Homological and homotopical aspects of torsion theories. {\it Mem. Amer. Math. Soc.}, 
vol. 188, n. 883, viii+207 pp., (2007).

\bibitem{CE} H. Cartan, S. Eilenberg. Homological algebra. {\it Princeton Landmarks in Mathematics. Princeton University Press},  xvi+390 pp., (1999).

\bibitem{D} S. E. Dickson. A torsion theory for abelian categories. {\it Trans. Amer. Math. Soc.} 121, 223-235, (1966). 
 
\bibitem{FLM} S. Fernandes, M. Lanzilotta, O. Mendoza. The $\phi$-dimension: a new homological measure. {\it Algebras and Representation Theory} vol. 18,(2), 463-476, (2015).


\bibitem{HL} F. Huard, M. Lanzilotta, Self-injective right artinian rings and Igusa-Todorov functions, {\it Algebras and Representation Theory}, 16 (3), pp. 765-770, (2012).

\bibitem{HLM1} F. Huard, M. Lanzilotta, O. Mendoza. An approach to the Finitistic Dimension Conjecture. {\it J. of Algebra} 319, 3918-3934, (2008).

\bibitem{IT} K. Igusa, G. Todorov. On the finitistic global dimension for artin algebras. {\it Representation of algebras and related topics}, Fields Institute Communications 45 (American Mathematical Society, Providence, RI),  201-204, (2005).

\bibitem{M} S. Michelena. Sobre un problema de clasificaci\'on y cohomolog\'{\i}a de Hochschild de extensiones locales, Tesis Doctor en Matem\'atica, Universidad Nacional del Sur, (1998).

\bibitem{Xu} D. Xu. Homological dimensions and strongly idempotent ideals. {\it J. Algebra} 414, 175-189, (2014).

\end{thebibliography}
\end{document}